\documentclass[A4paper,10pt]{amsart}
\usepackage[text={6.299213in,8.267717in},centering]{geometry}

\usepackage[french,main=english]{babel}
\usepackage{times, newtxmath,enumerate,dsfont} 
\usepackage{mathtools}
\usepackage{multirow}
\usepackage{graphicx}	
\usepackage{float}
\restylefloat{table} 
\usepackage{etoolbox}

\newcommand{\inv}{^{\raisebox{.2ex}{$\scriptscriptstyle-1$}}}   
  
\newcommand{\m}{M}
\newcommand{\C}{\mathcal{C}}
\newcommand{\dm}{\mathcal{D}(M)}
\newcommand{\sm}{\mathcal{S}(M)}

\usepackage[dvipsnames]{xcolor}   
\usepackage{xparse}
\usepackage{xr-hyper}
\usepackage[linktocpage=true,colorlinks=true,hyperindex,citecolor=RoyalBlue,linkcolor= RoyalBlue]{hyperref}   

\theoremstyle{plain}
\newtheorem{theorem}{Theorem}[section] 
\newtheorem{proposition}[theorem]{Proposition}
\newtheorem{lemma}[theorem]{Lemma}
\newtheorem{corollary}[theorem]{Corollary}
\theoremstyle{definition}
\newtheorem{definition}[theorem]{Definition}

\numberwithin{equation}{section}
\newtheorem{cremark}[theorem]{Concluding Remarks}

\begin{document}
\author{A. Goswami}
\address{
[1] Department of Mathematics and Applied Mathematics, University of Johannesburg, P.O. Box 524, Auckland Park 2006, South Africa. [2] National Institute for Theoretical and Computational Sciences (NITheCS), South Africa.}
\email{agoswami@uj.ac.za}

\title{Topological properties of some  classes  of submodules}




\begin{abstract}
We study the topology of a class of proper submodules and some of its distinguished subclasses and call them structure spaces. We give several criteria for the quasi-compactness of these structure spaces. We study $T_0$ and $T_1$ separation properties and characterize structure spaces in which nonempty irreducible closed subsets have unique generic points. We provide a sufficient condition for the connectedness of structure spaces. We prove that the structure spaces of proper submodules are spectral, and moreover, we characterize the spectral structure spaces of Noetherian modules. Finally, we discuss continuous functions between these spaces.
\end{abstract}
\maketitle

\section{Introduction and Preliminaries}

Since the introduction of prime submodules in \cite{D78}, much attention has been drawn to studying this class (see \cite{L12, L84, L95}). Following the existence of various classes of ideals of rings, it is natural to introduce similar classes of submodules, and hence we now have notions like maximal, prime, semiprime, primary, strongly irreducible, cyclic, and finitely generated submodules. As far as studying topologies is concerned, the main choice of class is prime submodules or its subclass of minimal prime submodules. However, it has been immediately observed that unlike prime ideals of a (commutative) ring, one cannot endow a Zariski topology on the class of prime submodules of a module. To make closed sets behave well with respect to finite unions, we need an additional condition on the module, and hence what we obtain is called a top module (see \cite{Z06}). An extensive study of various topologies on prime spectra (or their variations like minimal prime, coprime, fully prime) of submodules can be found in \cite{A11, A15, A10, AF14,AH12, AH19, AKF16, ZT00, AH11, AP17, BH08, L99, MMS97, MS06, HP19, T09, UTS15}.  

An attempt has been made in \cite{DG22} to obtain a unified approach to studying the topological properties of various classes of ideals of (commutative) rings, which is further carried on in \cite{FGS23} (see also \cite{G23}) to obtain spectral spaces out of these classes of ideals. These topological spaces were called ideal spaces. The goal of the present paper is twofold. First, we want to gain a better understanding of the properties of the ideal spaces of rings that also hold for modules, and second, we want to do this in a setting that allows us to handle wider classes of submodules that are not necessarily of ``prime type." These classes of submodules endowed with a topology shall be called structure spaces. Our topology coincides with Zariski topology when we consider the spectra of prime submodules of top modules. Although we have picked up some distinguished classes of submodules (see Definition \ref{dcsm}), our theory is in fact applicable for any subclass of the class of proper submodules. 

In this paper, by a module $\m$, we shall always mean an $R$-module, where  $R$ is a commutative ring with identity. A nonempty subset $N$ of a module $\m$ is called a \emph{submodule} of $\m$, if $N$ is closed under addition and scalar multiplication. A \emph{proper} submodule $N$ is such that $N\neq \m.$ If $N$ is a submodule of a module $\m$, then the \emph{annihilator ideal of $\m/N$ in $R$} is defined by \[(N : \m)=\{r\in R\mid r\m\subseteq N\}.
\] 
We now recall definitions of some distinguished subclasses of the class of proper submodules of a module. 

\smallskip
\begin{definition}\label{dcsm}
Let $\m$ be a module and let $N$ be a proper submodule of $\m$. Then $N$ is called
\begin{enumerate}
	
\item [$\bullet$] \emph{maximal} if there is no other proper submodule containing $N$;

\item[$\bullet$]  \emph{prime} \cite{D78} if $rm \in N$ with
$r \in R$ and $m \in \m$ imply that either $m \in N$ or $ r \in (N : M);$

\item[$\bullet$] \emph{semiprime} \cite{D78} if $N$ is an intersection of prime submodules of $M$;

\item[$\bullet$] \emph{extraordinary} \cite{MMS97} if $L$ and $K$ are semiprime submodules of $M$ with $L\cap K\subseteq N$, then either $L\subseteq N$ or $K\subseteq N$;

\item[$\bullet$] \emph{primary} \cite{AM69} if $rm \in N$ with
$r \in R$ and $m \in \m$ imply that either  $r^n \in (N : \m)$ for some $n\in \mathds{N}$, or $m \in N$;

\item[$\bullet$] \emph{radical} if
\[N=\sqrt{N}=\bigcap_{ N \subseteq P}\{P\mid P\;\text{is a prime submodule of}\;\m\};\]

\item[$\bullet$] \emph{strongly irreducible} \cite{KES06} if $L\cap K\subseteq N$ implies either $L\subseteq N$ or $K\subseteq N,$ for all submodules $L, K\in M$;

\item[$\bullet$] \emph{irreducible} \cite{M86} if $N\neq N_1\cap N_2$ for two submodules $N_1$ and $N_2$ of $\m$ with $N_i\neq N$;

\item[$\bullet$] \emph{completely irreducible} \cite{FHO06} if $N\neq \bigcap_{\lambda \in \Lambda}N_{\lambda}$ for submodules $\{N_{\lambda}\}_{\lambda\in \Lambda}N_1$  of $\m$ with $N_{\lambda}\neq N$ for all $\lambda \in \Lambda$;

\item[$\bullet$] \emph{minimal} \cite{AP17}  if $N\neq 0$ and $N$ contains no other nonzero submodules of $\m$;

\item[$\bullet$]  \emph{minimal prime} \cite{MS93} if $N$ is both a minimal and a prime submodule of $\m$.

\item[$\bullet$] \emph{cyclic} if $N$ is generated by a single  element of $\m$;

\item[$\bullet$] \emph{finitely generated} if $N$ is generated by a finite subset of $\m$.
\end{enumerate}
\end{definition}

Before we discuss topologies on distinguished classes of submodules in the next section, we first need some notation. By $\mathcal{S}(\m),$ we shall denote the set of all submodules of a module $M$. To denote the class of proper submodules of $M$ or any one of its  distinguished subclasses listed in Definition \ref{dcsm}, we shall use the notation $\mathcal{D}(M)$. It is clear that $M\notin \mathcal{D}(M)$ for all such classes of submodules. 

\smallskip
\section{Structure spaces}\label{ssms}

Our goal is to introduce a closure operation $\C$ on a distinguished class of submodules of $M$ in such a way that the closure operation $\C$ induces a topology on that class of submodules. The closure $\mathcal{C}$ is defined on a $\mathcal{D}(M)$ as follows.
\begin{equation}\label{vsl}
\mathcal{C}(N)=\{L\in \mathcal{D}(M)\mid S\subseteq L\}\qquad (N\in \mathcal{S}(M)).	
\end{equation} 

In the following lemma, we gather some elementary properties of $\C$ needed in the sequel. Since the proofs of these results are straightforward, we skip them.

\smallskip
\begin{lemma}\label{bpvs}
Let $M$ be a module. A closure operation $\mathcal{C}$ on a $\mathcal{D}(M)$ has the following properties.
\begin{enumerate}
	
\item  For any two $N, N' \in \dm $ with $N\subseteq N'$ imply that $\mathcal{C}(N)\supseteq \mathcal{C}(N');$
	
\item  \label{a} $\mathcal{C}(0)=\mathcal{D}(M)$ and $\mathcal{C}(M)=\emptyset;$

\item \label{b} $\bigcap_{\lambda\in \Lambda}\mathcal{C}(N_{\lambda})=\mathcal{C}\left( \sum_{\lambda \in \Lambda} N_{\lambda}\right)$ for all $N_{\lambda}\in \mathcal{S}(M)$ and for all $\lambda\in \Lambda ;$

\item\label{vnuvn} $\mathcal{C}(N)\cup \mathcal{C}(N') \subseteq \mathcal{C}(N\cap N')$ for all $N,$ $N'\in \mathcal{S}(M);$

\item\label{vnsvn} $\mathcal{C}(N)\supseteq \mathcal{C}(\sqrt{N})$ for all $N\in \sm$.
\end{enumerate}
\end{lemma}

In general, the sets $\{\mathcal{C}(N)\}_{N\in \sm}$ are not  closed under finite unions, and hence we may not have equality in Lemma \ref{bpvs}(\ref{vnuvn}). However, if  for any two $N, N'\in \sm$, there exists an $N''\in \sm$ such that $\C(N)\cup \C(N')=\C(N''),$ then obviously we obtain a equality in Lemma \ref{bpvs}(\ref{vnuvn}). A module with this property is called a \emph{top} module.  (see \cite[\textsection 2]{MMS97}). Therefore, in a top module, the collection $\{\mathcal{C}(N)\}_{N\in \sm}$ of subsets of any class $\dm$ of submodules induces the Zariski topology on $\dm$. It is clear from this discussion that unless $M$ is a top module, even for the class of prime submodules, the collection $\{\mathcal{C}(N)\}_{N\in \sm}$ does not induce the Zariski topology on it. This is the major difference in behaviour in topology endowed on ideals of a ring and on submodules of a module. 

For an arbitrary module $M$, using these subsets  of a distinguished class $\dm$ of submodules, we wish to construct a topology on $\mathcal{D}(M)$. 
Note that from Lemma \ref{bpvs}(\ref{a}), it follows that \[\bigcap_{N\in \mathcal{S}(M)} \mathcal{C}(N)=\emptyset,\] and hence by \cite[Theorem\,15\,A.13., p.\,254]{C66},  the collection $\{\mathcal{C}(N)\}_{N\in \mathcal{S}(M)}$ as  a closed subbasis generates a unique topology on  $\mathcal{D}(M)$. We denote this topology by $\tau_{\mathcal{D}(M)}$.  For the topological space $(\mathcal{D}(M), \tau_{\mathcal{D}(M)})$, we will write $\dm$ and call this topological space a \emph{structure space}.

\subsection{Compactness}

Next, we are interested in studying the quasi-compactness of a structure space, $\mathcal{D}(M)$. Although, from Lemma \ref{bpvs}(\ref{a}), we have $N=M$ implies that $\C(N)=\emptyset$, but the converse may not be true for all distinguished classes of submodules, and what we want is to use this converse property to prove the quasi-compactness of structure spaces. However, if $M$ is finitely generated, then the converse holds for any $\mathcal{D}(M)$.
Since our topology is generated by subbasis closed sets, in the proof of the following theorem we shall rely on the Alexander Subbasis Theorem. Note that our result generalizes Theorem 2.13 of \cite{ZT00}.

\smallskip
\begin{theorem}\label{comp} 
If $M$ is a finitely generated module, then every structure space is quasi-compact. 
\end{theorem} 

\begin{proof}   
To show the compactness, thanks to Alexander Subbasis Theorem, it suffices to prove the finite intersection property only for subbasis closed sets. Let $\bigcap_{\lambda\in \Lambda}K_{ \lambda}=\emptyset$ for a family  $\{K_{ \lambda}\}_{\lambda \in \Lambda}$ of subbasis closed sets of a structure space $\mathcal{D}(M)$. We claim that $\bigcap_{ i\,=1}^{ n}K_{ \lambda_i}=\emptyset,$ for some finite subset $\{\lambda_1, \ldots, \lambda_n\}$ of $\Lambda$. Let $\{N_{ \lambda}\}_{\lambda \in \Lambda}\in \sm$  such  that $K_{ \lambda}=\mathcal{C}(N_{ \lambda})$ for all $\lambda \in \Lambda,$  By Lemma \ref{bpvs}(\ref{b}), this implies that $\mathcal{C}\left(\sum_{\lambda \in \Lambda}N_{ \lambda}\right)=\emptyset.$ Since $M$ is finitely generated, $ \sum_{\lambda \in \Lambda}N_{ \lambda}=M$, and also there exists a finite subset of $\Lambda$ for which this equality holds.
\end{proof}  

Another sufficient condition for the quasi-compactness of a structure space $\dm$ is the inclusion of all maximal submodules of $M$ in $\dm$. However, a module may not have any maximal submodules, which can be guaranteed to exist if we consider a module over a perfect ring (see \cite{B60}). Note that the basic structure of argument in the proof of the following theorem is similar to that of Theorem \ref{comp}.

\smallskip
\begin{theorem}\label{csb}
If $M$ is a module over a perfect ring and if a class $\dm$ of submodules of $M$ contains all maximal submodules of $M$, then the structure space $\dm$ is  quasi-compact. 
\end{theorem}

\begin{proof}
Let $\{\mathcal K_{ \lambda}\}_{\lambda \in \Lambda}$ be a family of subbasis closed sets of $\dm$  such that $\bigcap_{\lambda\in \Lambda}\mathcal K_{ \lambda}=\emptyset.$ This implies that
$\mathcal K_{ \lambda}=\mathcal{C}(N_{\lambda})$ for some ideals $N_{\lambda}$ of $M$, and 
$$\bigcap_{\lambda \in \Lambda}\mathcal{C}(N_{\lambda})=\mathcal{C}\left(\sum_{\lambda \in \Lambda} N_{\lambda}  \right)=\emptyset.$$
If $\sum_{\lambda \in \Lambda} N_{\lambda}  \neq M,$ then we must have a maximal ideal $L$ of $M$  such that $\sum_{\lambda \in \Lambda} N_{\lambda}  \subseteq L.$ Moreover, 
$$ N_{\lambda}  \subseteq \sum_{\lambda \in \Lambda} N_{\lambda}  \subseteq L,$$
for all $\lambda \in \Lambda.$ Therefore $L\in \mathcal{C}(N_{\lambda})=\mathcal K_{\lambda}$ for all $\lambda \in \Lambda$, a contradiction of our assumption. Hence $\sum_{\lambda \in \Lambda} N_{\lambda}=M,$ and the identity $1\in \sum_{\lambda \in \Lambda} N_{\lambda}.$ This implies the existence of a finite subset $\{\lambda_{\scriptscriptstyle 1}, \ldots, \lambda_{\scriptscriptstyle n}\}$ of $\Lambda$ such that $1= \sum_{i=1}^n x_{\lambda_i}$ (where $x_{\lambda_i}\in N_{\lambda_i}$), and hence  $M= \sum_{i=1}^n N_{\lambda_i},$ which establishes the finite intersection property. Therefore, $\dm$ is quasi-compact by Alexander Subbase Theorem.
\end{proof} 

\smallskip
\begin{corollary}
The structure spaces of maximal, prime, semiprime, extraordinary, strongly irreducible, primary, irreducible, completely irreducible, radical submodules are quasi-compact.
\end{corollary}

Notice that in Theorem
\ref{csb}
the containment of all maximal submodules to a class $\dm$ of submodules is a sufficient condition for quasi-compactness of the structure space $\dm$. However, for the class of finitely generated submodules of a module, for instants, it is also a necessary condition.

\smallskip
\begin{proposition}
Let $M$ be a module over a perfect ring.	If the structure space $\dm$ of finitely generated (proper) submodules is quasi-compact, then $\dm$ contains all maximal submodules of $M.$ 
\end{proposition}

\begin{proof}
Let $L$ be a maximal submodule of the module $M$ such that $L$ is not finitely generated. Let us consider the collection of closed subspaces: $$\mathcal K=\left\{\mathcal{C}(\langle x\rangle)\cap \dm\mid x\in L \right\}.$$  We claim that $\cap \mathcal K=\emptyset$. If not, let $T\in \cap \mathcal K$. Then $T$ is finitely generated and $L\subseteq T$. Since $L$ is not a finitely generated submodule, we must have $T\supsetneq L,$ and that implies $T=M$, which contradicts the fact that $\dm$ consists of proper submodules. But clearly $\mathcal K$ has the finite intersection property and hence $\dm$ is not quasi-compact.
\end{proof}

Since, in general, a submodule of a finitely generated module need not be finitely generated, the following result provides another sufficient condition for the quasi-compactness of a structure space, and this is an extension of Proposition 3.5 of \cite{FGS23}.

\smallskip
\begin{theorem}\label{niqc}
If $M$ is a Noetherian module then every structure space $\dm$ is quasi-compact.
\end{theorem}   

\begin{proof}
Consider a collection $\{\dm\cap \mathcal{C}(N_{\lambda}) \}_{\lambda\in \Omega}$ of subbasis closed sets of $\dm$ with the finite intersection property. By assumption, the submodule $T=\sum_{\lambda\in \Omega}N_{\lambda}$ is finitely generated, say $T=(\alpha_1,\ldots,\alpha_n)$. For  every $1\leqslant j\leqslant n$, there exists a finite subset $\Lambda_j$ of $\Omega$ such that $\alpha_j\in \sum_{\lambda\in \Lambda_j}N_{\lambda}$. Thus, if  $\Lambda:=\bigcup_{j=1}^n\Lambda_j$, it immediately follows that $T=\sum_{\lambda\in \Lambda}N_{\lambda}$. Hence we have
\begin{align*} 
\bigcap_{\lambda\in \Omega}\left(\dm\cap \mathcal{C}(N_{\lambda})\right)&=\dm\cap \mathcal{C}(T)\\&=\dm\cap  \mathcal{C}\left(\sum_{\lambda\in \Lambda}N_{\lambda} \right)\\&= \bigcap_{\lambda\in \Lambda}\left(\dm\cap \mathcal{C}(N_{\lambda})\right) \neq \emptyset,
\end{align*}
since $\Lambda$ is finite and $\{\dm\cap \mathcal{C}(N_{\lambda}) \}_{\lambda\in \Omega}$ has the finite intersection property. Then the conclusion follows by the Alexander Subbasis Theorem. 
\end{proof}

\smallskip
\begin{corollary}
If $M$ is a Noetherian module then every structure space $\dm$ is  Noetherian.
\end{corollary} 

\subsection{Separation properties} We now focus on separation properties of structure spaces. Since for any two distinct elements $L_1$ and $L_2$ of a distinguished class of submodules implies that $\C(L_1)\neq \C(L_2)$, we immediately have the following separation property, which generalizes Proposition 2.4 of \cite{ZT00}.

\smallskip
\begin{proposition}\label{t0}
Every structure space is  a $T_0$-space. 
\label{ct0t1}  
\end{proposition}

Before we embark on the existence of generic points of structure spaces, we recall two topological notions that we need. A closed subset $S$ of a space $X$ is called \emph{irreducible} if $S$ is not the union of two properly smaller closed subsets of $X$. An element $s$ of $S$ is called a \emph{generic point} of $S$ if $S = \overline{\{s\}},$ where $\overline{\{s\}}$ is the closure of $\{s\}$. 

If $M$ is a top module and $\mathcal{D}(M)$ is the class of prime submodules of $M$ endowed with the Zariski topology, then it is shown in \cite[Proposition 3, Corollary 1]{Y11} that a nonempty closed subset $X$ of $\mathcal{D}(M)$ is irreducible if and only if  the intersection of all prime submodules belonging to $X$ is itself a prime submodule of $M$. For an arbitrary structure space $\dm$, we do not have a similar characterization. However, in Lemma \ref{irrc} below, we shall identify a class of irreducible closed subsets of a structure space $\dm$, which is in fact good enough to play good roles in 

\begin{itemize}
	
\item[$\bullet$] characterizing $T_1$ structure spaces (Theorem \ref{t1}),

\item[$\bullet$] proving the structure space of proper submodules is spectral (Theorem \ref{pss}), and

\item[$\bullet$] studying connectedness of structure spaces (Theorem \ref{conis}).
\end{itemize}

\smallskip
\begin{lemma}\label{irrc}
The subsets $\{\mathcal{C}(N)\mid N\in \mathcal{C}(N)\}$ of a structure space $\dm$ are irreducible. 
\end{lemma} 

\begin{proof} 
Let $\overline{N}$ be the closure of the submodule $N$ of a module $M$. We actually show more, namely, $\overline{N}=\mathcal{C}(N)$ for every $N\in\mathcal{C}(N)$. It is evident that $\overline{N}\subseteq \mathcal{C}(N)$. 
To have the other half of the inclusion, let us first consider the trivial case, that is, when $\overline{N}= \mathcal{D}(M)$. Since
\[
\mathcal{D}(M)=\overline{N}\subseteq \mathcal{C}(N)\subseteq \mathcal{D}(M),
\]
we have the equality. Now, suppose $\overline{N}\neq  \mathcal{D}(M)$.
Since $\overline{N}$ is a closed set, it has the following representation 
\[
\overline{N}={\bigcap_{\lambda\in\Lambda}}\left({\bigcup_{ i\,=1}^{ n_\lambda}}\mathcal{C}(N_{\lambda i})\right),
\]
for some index set $\Lambda$ and submodules  $N_{\lambda 1},\dots, N_{\lambda n_\lambda}$ of $M$. 
By hypothesis, this implies ${\bigcup_{ i\,=1}^{ n_\lambda}}\mathcal{C}(N_{\lambda i})\neq \emptyset$, and  hence $N\in   {\bigcup_{ i\,=1}^{ n_\lambda}}\mathcal{C}(N_{\lambda i})$, for each $\lambda$. From these, we conclude that for each $\lambda$, \[\mathcal{C}(N)\subseteq {\bigcup_{ i=1}^{ n_\lambda}}\mathcal{C}(N_{\lambda i}),\] and hence, $\mathcal{C}(N)\subseteq \overline{N},$ as required.	
\end{proof}     

\smallskip
\begin{corollary}\label{spiir}
Every non-empty subbasis closed subset of structure space of proper submodules is irreducible.
\end{corollary} 

One immediate application of Lemma \ref{irrc} is the following characterization of $T_1$-structure spaces. It is known (see \cite[Theorem 3]{Y11}) that in a top module, the class of prime submodules endowed with Zariski topology is a $T_1$-space if and only if every prime submodule is maximal in the class of prime submodules. The following theorem generalizes this result.

\smallskip
\begin{theorem}\label{t1}
A structure space is a $T_1$-space if and only if it is contained in the class of structure space of maximal submodules. 
\end{theorem}  

\begin{proof} 
If $L\in\mathcal{D}(M)$, then by Lemma \ref{irrc}, $\overline{L}=\mathcal{C}(L)$. Suppose $N$ is a maximal submodule of $M$ such that $L\subseteq N$. Let $\mathcal{D}(M)$ be a $T_1$-space. Then   \[N\in 	\mathcal{C}(L)=\overline{L}= \{L\},\] which implies $N=L$.  
Conversely, let $\dm$ be a class of submodules of $M$ contained in the class of maximal submodules of $M$. Then for every maximal submodule $N$ of $M$, we have $\mathcal{C}(N)=\{N\}$, which implies $N\in \mathcal{C}(N)$. Hence, by Lemma \ref{irrc}, $\overline{N}=\{N\}$. This proves that the structure space of the class of maximal submodules of $M$ is a $T_1$-space and so is $\dm$.
\end{proof} 

As a consequence of the above theorem, we obtain a necessary condition for a Noetherian module to be Artinian. Recall that a module is called \emph{Noetherian} if every ascending chain of submodules of it is eventually
stationary, whereas an Artinian module is a module that satisfies the descending chain condition on its poset of submodules.

\smallskip
\begin{corollary}
If $\mathcal{D}(M)$ is a discrete structure space of a Noetherian module $M$, then $M$ is Artinian.  
\end{corollary}

Recall that a topological space is \emph{sober} if every nonempty irreducible closed subset of it has a unique generic point. Sobriety is one of separation properties and also a component of the definition of a spectral space, which we shall discuss later. In \cite[Theorem 3.3]{AH11}, it is proved that the class of prime submodules of a top module endowed with Zariski topology is sober. We shall provide a characterization of structure spaces that have this property and this characterization generalizes the results Corollary 1 of \cite{Y11}, Theorem 3.3 of \cite{AH11}, and Theorem 2.6(2)(3) of \cite{ZT00}.  We first need a technical auxiliary lemma whose proof is easy to verify.

\smallskip
\begin{lemma}
\label{axaa} If $ N^{\omega}:=\bigcap\{L\mid L\in \mathcal{C}(N)\}$ for $N\in \sm$, then 
$N=N^{\omega}$, whenever $N\in\mathcal{D}(M).$
\end{lemma}

\smallskip
\begin{theorem}\label{sob}  
Every non-empty irreducible subbasis closed set $\mathcal{C}(N)$ of a structure space $\mathcal{D}(M)$ has a unique generic point if and only if $\mathcal{C}(N)$ contains $N^{\omega}.$ 
\end{theorem} 

\begin{proof}
Let $\mathcal{C}(N)$ be a nonempty irreducible subbasis closed subset of a structure space $\mathcal{D}(M)$ such that $\mathcal{C}(N)$ has a unique generic point $L$. This implies that 
\[\mathcal{C}(N)=\overline{L}=\mathcal{C}(L).\] By Lemma \ref{axaa}, this implies that $L=N^{\omega}\in \mathcal{D}(M)$. 
The proof of the converse is more intriguing. Let every nonempty irreducible subbasis  closed set $\mathcal{C}(N)$ contains $N^{\omega}$ and let $E$ be an irreducible closed subset of $\mathcal{D}(M)$. Then \[E=\bigcap_{i\in I}\left( \bigcup_{\alpha=1}^{n_i} \mathcal{C}(N_{i\alpha})\right),\] for some  submodules $N_{ij}$ of $M$. Since $E$ is irreducible, for every $i\in I,$ there exists a submodule $N_i$ of $M$ such that \[E\subseteq \mathcal{C}(N_i)\subseteq \bigcup_{\alpha=1}^{n_i} \mathcal{C}(N_{i\alpha}),\] and thus  \[E=\bigcap_{i\in I}\mathcal{C}(N_i)=\mathcal{C}\left(\sum_{i\in I}N_i \right)=\mathcal{C}\left(\sum_{i\in I}N_i \right)^{\omega}.\] Since $\left(\sum_{i\in I}N_i \right)^{\omega}\in \mathcal{D}(M)$, we have $E=\overline{\left(\sum_{i\in I}N_i \right)^{\omega}}$. The uniqueness part of the claim follows from Proposition \ref{t0}, and this proves the theorem.
\end{proof}

\smallskip
\begin{corollary}
Every nonempty irreducible closed subset of structure spaces of proper, prime, minimal prime submodule has a unique generic point.
\end{corollary}

To show sobriety of the structure space $\dm$ of strongly irreducible submodules of a module, our argument goes through multiplicative lattices. Since $\dm$ is a multiplicative lattice with product as intersection, every strongly irreducible submodule becomes a prime element of this lattice, and hence by \cite[Lemma 2.6]{FFJ22}, every nonempty irreducible closed subset of $\dm$ has a unique generic point, where once again the uniqueness of generic points follows from Proposition \ref{t0}. The structure spaces of semiprime and extraordinary submodules are also sober because of the same reason. The following proposition gathers all these conclusions.

\smallskip
\begin{proposition}\label{sirrs} 
Every nonempty irreducible closed subset of structure spaces of strongly irreducible, semiprime, and extraordinary submodules has a unique generic point.
\end{proposition}

Our next goal is to prove the spectrality of the structure space of a specific class of submodules of a module, and this is an extension of Theorem 2.9 of \cite{G23}.

\smallskip
\begin{theorem}\label{pss}
The structure space of proper submodules of a module is spectral.
\end{theorem}

\subsubsection{Spectral spaces}
Recall that a topological space is called \emph{spectral} (see \cite{H69}), if it quasi-compact,  admitting a basis
of quasi-compact open subspaces that are closed under finite intersections, and if every nonempty irreducible closed subset of them has a unique generic point. Our proof of Theorem \ref{pss} is self-contained in the sense that it is constructible topology-independent. Moreover, we have adapted a technique that avoids the checking of the existence of a
basis of quasi-compact open subspaces that are closed under finite intersections, which rather follows from the following key lemma.

\smallskip
\begin{lemma}\label{cso}
Let $X$ be a spectral space and let $S$ be a subspace of $X$ having the following properties:
\begin{enumerate}
\item $S$ is quasi-compact.

\item $S$ is an open subspace of $X$.

\item Every nonempty irreducible closed subset of $S$ has a unique generic point.
\end{enumerate} 
Then $S$ is a spectral space.
\end{lemma}

\begin{proof}
Since $S$ is quasi-compact and sober, all we need is to prove that the set $\mathcal{B}_{\scriptscriptstyle S}$ of quasi-compact open subsets of $S$ forms a basis of a topology that is closed under finite intersections. It is easy to check the following fact. A subset of $S$ is open in $S$ if and only if it is open in $X$, and therefore, a subset of $S$ belongs to $\mathcal{B}_{\scriptscriptstyle S}$ if and only if it belongs to $\mathcal{B}_{\scriptscriptstyle X}.$

Let $O$ be an open subset of $S$. Since $O$ is also open in $X$, we have $O=\cup\, \mathcal{U},$ for some subset $\mathcal{U}$ of $\mathcal{B}_{\scriptscriptstyle X}.$ But each element of $\mathcal{U}$ being a subset of $O$ is a subset of $S$, and it belongs to $\mathcal{B}_{\scriptscriptstyle S}.$ Therefore every open subset of $S$ can be presented as a union of quasi-compact open subsets of $S$. Now it remains to prove that $\mathcal{B}_{\scriptscriptstyle S}$ is closed under finite intersections, but this immediately follows from the fact that $\mathcal{B}_{\scriptscriptstyle X}$ is closed under finite intersections. 
\end{proof}
\smallskip

\noindent 
\textit{Proof of Theorem \ref{pss}}.
In order to apply Lemma \ref{cso} for the structure space $\dm$ of proper submodules of a module, we need to achieve the following:
 
\begin{enumerate}
\item The structure space $\sm$ of all submodules of a module $M$ is spectral.

\item \label{oco} $\dm$ is quasi-compact.

\item\label{tso} Every nonempty irreducible closed subset of $\dm$ has a unique generic point.

\item\label{top}  $\dm$ is an open subspace of the structure space $\sm$.
\end{enumerate}
Now, (2) and (3) respectively follow from Proposition \ref{comp} and Proposition \ref{sob}, whereas (1) follows from the application of Theorem 4.2 from \cite{P94} on the algebraic lattice $\sm.$ Therefore, what remains is to show (4). Since $M\in\sm,$ by Lemma \ref{irrc}, we have $M=\mathcal{C}(M)=\overline{M},$ and therefore, $\{M\}$ is closed. This implies that $\dm$ is open, and this completes the proof. 
\hfill \qedsymbol
\smallskip

If our modules are Noetherian and a structure space is sober, then, thanks to Theorem \ref{niqc}, we obtain further examples of spectral spaces.

\smallskip
\begin{corollary}\label{cor}
Let $R$ be a Noetherian module and let $\dm$ be a structure space. Then $\dm$ is spectral if and only if it is sober. 
\end{corollary}

\subsection{Connectedness}
Now we turn our attention to the connectedness of structure spaces. The usual notion of connectedness does not have a direct formulation in terms of subbasis closed sets. We, therefore, introduce a different kind of connectedness (in terms of subbasis closed sets) of structure spaces that is related to the conventional notion of connectedness.
We say a closed subbasis $\mathcal{S}$ of a topological space $X$ \emph{strongly disconnects} $X$  if there exist two non-empty subsets $A$ and $B$ of $\mathcal{S}$ such that $X=A\cup B$ and $A\cap B=\emptyset$.

From the above definition, it is clear that if some closed subbasis strongly disconnects a space, then the space is disconnected. Moreover, if a space is disconnected, then some closed subbasis strongly disconnects it. But this does not imply that every closed subbasis strongly disconnects the space. However, we have the following result, analogous to Theorem 3.21 and Lemma 3.22 of \cite{DG22}.

\smallskip
\begin{lemma}\label{th1}
If a quasi-compact space is disconnected, then every closed basis that is closed under finite intersections strongly disconnects the space. Moreover, for any structure space $\mathcal{D}(M)$, the closed basis is closed under binary intersections.
\end{lemma}

Applying Lemma \ref{th1}, we immediately have the following result.

\smallskip
\begin{theorem}\label{cor1}
Suppose $\mathcal{D}(M)$ is a quasi-compact space. Then  $\mathcal{D}(M)$ is  disconnected if and only if the closed basis strongly disconnects $\mathcal{D}(M)$. 
\end{theorem}   

A consequence of Lemma \ref{irrc} is the following sufficient condition for a structure space to be connected. 

\smallskip
\begin{theorem}\label{conis}
If the zero submodule is in  $\mathcal{D}(M),$  then the structure space $\mathcal{D}(M)$ is connected.
\end{theorem}  

\begin{proof}
By Lemma \ref{bpvs}(\ref{a}), we have $\mathcal{D}(M)=\mathcal{C}(0)$, and since irreducibility implies connectedness, the claim now follows from Lemma \ref{irrc}. 
\end{proof}

\smallskip
\begin{corollary}
The structure spaces of proper, finitely generated, cyclic submodules are connected. 
\end{corollary} 

\subsection{Continuity}
Once we have structure spaces of submodules, it is natural to study continuous functions between them. Let $\phi\colon M\to M'$ be a module homomorphism. Likewise, for prime spectra (endowed with Zariski topologies) of rings, it is natural to expect a continuous Likewise, for prime spectra (endowed with Zariski topologies) of rings, it is natural to expect a continuous function $\phi_!\colon \mathcal{D}(M')\to \mathcal{D}(M)$ between ``similar type'' submodule spaces. It is indeed so, and we shall discuss that in the next proposition. However, in the case of rings, under a ring homomorphism, the inverse image of a prime spectrum is a prime spectrum, which even fails to hold for maximal ideals of rings. Therefore, we can expect the same problem when we deal with various classes of submodules of a module. So, we need to consider only those module homomorphisms for which such preservation holds. We call this the \emph{contraction property} of a module homomorphism.

To show a function is continuous, it suffices to show that the inverse image of subbasis closed sets are subbasis closed sets. Because  our topology is built of subbasis closed sets, we shall use this continuity criteria in our proofs.

\smallskip
\begin{proposition}\label{conmap}
Let $\mathcal{D}(M)$ be a structure space and $\phi\colon M\to M'$ be a module homomorphism having the contraction property with respect to $\mathcal{D}(M)$.   Let $N'\in\mathcal{D}(M').$ Define a function  $\phi_!\colon  \mathcal{D}(M')\to \mathcal{D}(M)$ by  $\phi_!(N)=\phi\inv(N)$. Then
\begin{enumerate}
		
\item \label{contxr} the function  $\phi_!$ is    continuous.
		
\item \label{shcs} If $\phi$ is  surjective, then the structure space $\mathcal{D}(M')$ is homeomorphic to the closed subspace $\mathcal{C}(\mathrm{ker}\,\phi)$ of the structure space $\mathcal{D}(M).$
		
\item \label{imde} The image  $\phi_!(\mathcal{D}(M'))$ is dense in $\mathcal{D}(M)$ if and only if $$\mathrm{ker}\,\phi\subseteq \bigcap_{L\in \dm}L;$$
\end{enumerate}
\end{proposition}

\begin{proof}      
(1) Let $\mathcal{C}(N)$ be a subbasis closed set of a structure  space $\mathcal{D}(M).$ Since $$\phi_!\inv(\mathcal{C}(N)) =\{N'\in \mathcal{D}(M')\mid \phi(N)\subseteq N'\}=\mathcal{C}(\langle \phi(N)\rangle),$$ we have the desired continuity of the function $\phi_!$.     

(2) We observe that     
$\phi_!(N')\in \mathcal{C}(\mathrm{ker}\,\phi)$, and this implies  $\mathrm{im}\,\phi_!=\mathcal{C}(\mathrm{ker}\,\phi).$  
Since for all $N'\in \mathcal{D}(M'),$ \[\phi(\phi_!(N'))=N'\cap \mathrm{im}\,\phi=N',\] the function $\phi_!$ is injective. To show that $\phi_!$ is a closed function, first we claim that under a module homomorphism, inverse image of a subbasis closed set is a subbasis closed. Indeed, for any subbasis closed set  $\mathcal{C}(N)$ of  $\mathcal{D}(M')$, we have \[\phi_!(\mathcal{C}(N))=\phi\inv\{ L'\in \mathcal{D}(M')\mid N\subseteq   L'\}=\mathcal{C}(\phi\inv(N)).\]   If $\mathcal{K}$ is a closed subset of $\mathcal{D}(M')$, then there exist an index set $\Omega$ and submodules $\{N_{i\alpha}\}_{1\leqslant i \leqslant n_{\alpha}, \alpha \in \Omega}$ of $M'$ such that \[\mathcal{K}=\bigcap_{ \alpha \in \Omega} \left(\bigcup_{ i \,= 1}^{ n_{\alpha}} \mathcal{C}(N_{ i\alpha})\right).\] Then, applying $\phi_!$ on $\mathcal{K},$ we have \[\phi_!(\mathcal{K})=\phi\inv \left(\bigcap_{ \alpha \in \Omega} \left(\bigcup_{ i = 1}^{ n_{\alpha}} \mathcal{C}(N_{ i\alpha})\right)\right)=\bigcap_{ \alpha \in \Omega} \bigcup_{ i = 1}^{ n_{\alpha}} \mathcal{C}(\phi\inv(N_{ i\alpha})),\] a closed subset of  $\mathcal{D}(M).$ Since  $\phi_!$ is   continuous by (1), this completes the proof.

(3) We claim that $\overline{\phi_!(\mathcal{C}(N'))}=\mathcal{C}(\phi\inv(N')),$ for all $N'\in \mathcal{S}(M').$ To this end, let $L\in \phi_!(\mathcal{C}(N')).$ From this, we obtain $\phi(L)\in \mathcal{C}(N'),$ or, equivalently, $L\in \mathcal{C}(\phi\inv(N')).$ From the fact that $\phi\inv(\mathcal{C}(N'))=\mathcal{C}(\phi\inv(N')),$ we then obtain the other half of the inclusion. Since $$\overline{\phi_!(\mathcal{D}(M'))}=\mathcal{C}(\phi\inv(0))=\mathcal{C}(\mathrm{ker}\,\phi),$$ the closed subspace $\mathcal{C}(\mathrm{ker}\,\phi)$ is equal to $\mathcal{D}(M)$ if and only if $\mathrm{ker}\,\phi\subseteq \cap_{L\in \mathcal{D}(M)}L.$ 
\end{proof}  
  
\smallskip
\begin{corollary}
The structure space $\mathcal{D}(M/N)$ is homeomorphic to the closed subspace
$\mathcal{C}(N)$ of $\mathcal{D}(M)$.
\end{corollary}

\smallskip
\begin{cremark}
We conclude with some programs for further study.

(A) The present work can be generalized further to study the topological properties of various classes of subsemimodules of semimodules. With some routine changes of terminologies and assumptions, one may extend some of the results obtained in \cite{HPH21} for the prime spectra of semimodules to distinguished classes of subsemimodules, similar to that which we have considered here. In Table \ref{tabB}, we provide some examples of these generalizations.

\begin{table}[H]
\begin{center}
\begin{tabular}{|c|c|}
\hline
\small\textbf{\cite{HPH21}} & \small\textbf{Present paper}\\
\hline\hline 
\small Lemma 3.1 &\small Lemma 2.1 \\
\hline
\small Theorem 3.1 &\small Proposition 2.8\\
\hline
\small Theorem 3.2 &\small Theorem 2.11\\
\hline
\small Theorem 3.3&\small Lemma 2.9 \\
\hline
\small Corollary 3.1& \small Theorem 2.14\\
\hline 
\small Theorem 3.5 &\small Theorem 2.2\\
\hline
\end{tabular}
\end{center}
\caption{Modules vs. Semimodules}
\label{tabB}
\end{table}

(B) In Theorem \ref{pss} and Corollary \ref{cor}, we have seen examples of spectral structure spaces. With some additional assumptions on modules, it would be worth it to obtain further examples of structure spaces that are spectral.

(C) Due to the lack of subtractivity in semirings, the theory of ideals is significantly different than that of rings. To minimize this gap, the notion of \emph{k-ideals} (also known as \emph{subtractive} ideals) has been introduced in \cite{H58}, which further generalized it for semimodules, called \emph{subtractive} submodules (see \cite{G99}). To best of author's knowledge, no topological studies have been made on this special class of subsemimodules.

\end{cremark}

\smallskip
\providecommand{\bysame}{\leavevmode\hbox to3em{\hrulefill}\thinspace}
\providecommand{\MR}{\relax\ifhmode\unskip\space\fi MR }
\providecommand{\MRhref}[2]{%
\href{http://www.ams.org/mathscinet-getitem?mr=#1}{#2}
}
\providecommand{\href}[2]{#2}

\end{document}